\documentclass{amsart}

\theoremstyle{plain}
\newtheorem{theorem}{Theorem}[section]
\newtheorem{corollary}[theorem]{Corollary}
\newtheorem{proposition}[theorem]{Proposition}
\newtheorem{lemma}[theorem]{Lemma}
\newtheorem*{theorem*}{Theorem}

\theoremstyle{definition}
\newtheorem{definition}[theorem]{Definition}

\newtheorem{remark}[theorem]{Remark}

\DeclareMathOperator{\USp}{USp}
\DeclareMathOperator{\Sp}{Sp}

\usepackage[initials]{amsrefs}

\begin{document}

\title[The summatory function in function fields]{The summatory function of the M\"obius function in function fields}
\author{Byungchul Cha}
\address{Muhlenberg College, 2400 Chew st, Allentown, PA 18104}
\date{\today}


\begin{abstract}
We study the growth rate of the summatory function of the M\"obius function in the context of an algebraic curve over a finite field. Our work shows a strong resemblance to its number field counterpart, as described by Ng in 2004. We find an expression for a bound of the summatory function, which becomes sharp when the zeta zeros of the curve satisfy a certain linear independence property. Extending a result of Kowalski in 2008, we prove that most curves in the family of universal hyperelliptic curves satisfy this property. Then, we consider a certain geometric average of such bound in this family, using Katz and Sarnak's reformulation of the equidistribution theorem of Deligne. Lastly, we study an asymptotic behavior of this average as the family gets larger by evaluating the average values of powers of characteristic polynomials of random unitary symplectic matrices. 
\end{abstract}

\maketitle

\section{Introduction}

Recall that the M\"obius function $\mu(n)$ is defined for any positive integer $n$ by
\[ \mu(n) := \begin{cases} 1 & \text { if } n = 1, \\
0 & \text{ if  $n$ is not square free,} \\
(-1)^t & \text{ if $n$ is a product of $t$ distinct primes.}
\end{cases}\]
Let $M(x)$ be its summatory function
\[ M(x) := \sum_{n\le x} \mu(n).\]
Mertens's conjecture \cite{Mer97} states that the inequality
\begin{equation}\label{eq:MertensConjecture}
|M(x)| < \sqrt x
\end{equation}
holds for all $x>1$. This conjecture was disproved by Odlyzko and te Riele \cite{OtR85} in 1985. Still, understanding the growth of $M(x)$ remains as a subject of intensive investigation by many authors in analytic number theory. For example, see \cite{MM09} and \cite{Sou09} for some recent results. Relevant to us is a paper \cite{Ng04} of Ng, where he gives certain conditional results on the growth of $M(x)$, using the techniques of Rubinstein and Sarnak in \cite{RS94}. In particular, Ng presents a probabilistic argument supporting the conjecture of Gonek, which states that the magnitude of $M(x)/\sqrt x$ grows roughly as $(\log\log\log x)^{5/4}$. More precisely, it asserts that there exists a number $B>0$ such that 
\begin{equation}\label{eq:GonekConjecture}
\limsup_{x\to \infty} \frac{M(x)}{\sqrt x (\log \log\log x)^{5/4}} = B, \text{ and }
\liminf_{x\to \infty} \frac{M(x)}{\sqrt x (\log \log\log x)^{5/4}} = -B.
\end{equation}

In this paper, we try to construct a function field analog of Ng's work and examine several issues that arise from this attempt. 
This is motivated by the present author's earlier paper \cite{Cha08}, where a function field version of Rubinstein and Sarnak's work is established. 
To describe our results in more detail, we fix some notations first. Let $C$ be a nonsingular projective curve of genus $g$ defined over a finite field $\mathbb{F}_q$ of characteristic $p>2$ with $q$ elements. Define the M\"obius function $\mu_{C/\mathbb{F}_q}(D)$ of $C/\mathbb{F}_q$ for all effective divisors $D$ of $C$ in the obvious way,
\[ \mu_{C/\mathbb{F}_q}(D) := \begin{cases} 1 & \text { if } D=0, \\
0 & \text{ if a prime divisor divides $D$ with its order at least 2,} \\
(-1)^t & \text{ if $D$ is a sum of $t$ distinct prime divisors.}
\end{cases}\]
Also, define the summatory function
\[ M_{C/\mathbb{F}_q}(X) := \sum_{\deg D \le X} \mu_{C/\mathbb{F}_q}(D),\]
for all positive integers $X$. 

The starting point is an asymptotic formula for $M_{C/\mathbb{F}_q}(X)$ as $X\to\infty$, which is given in Proposition \ref{thm:EquationMX}.
Roughly speaking, this formula says that $M_{C/\mathbb{F}_q}(X) = O(X^{r-1}q^{X/2})$, where $r$ is the maximum order of all inverse zeros for $C/\mathbb{F}_q$ (see \eqref{eq:RiemannHypothesis} for the definition of inverse zeros and their orders.) 
From this, if all inverse zeros are simple, we deduce in Corollary \ref{thm:MertensConjectureFunctionField} that the quantity
\[
B(C/\mathbb{F}_q):= \limsup_{X\to\infty}\frac{M_{C/\mathbb{F}_q}(X)}{q^{(X+1)/2}}
\]
exists as a finite number. This could be regarded as a (weak) function field analog of the inequality \eqref{eq:MertensConjecture}.

However, it is obvious that the boundedness of $M_{C/\mathbb{F}_q}(X)/q^{X/2}$, when all inverse zeros are simple, stems from the fact that there are only finitely many inverse zeros for any given $C$.
So, rather than studying $B(C/\mathbb{F}_q)$ for a single curve $C$, it would be interesting to find the \emph{average} of $B(C/\mathbb{F}_q)$ over a family $\mathcal{F}$ of curves whose genus $g$ is large. 
Instead, what we would like to do in this paper is find the average of $B(C/\mathbb{F}_q)$ over $\mathcal{F}$ with the \emph{scalar field} $\mathbb{F}_q$ \emph{growing larger}. The advantage of dealing with this \emph{geometric average} is that this set-up enables us to use the powerful tool of Katz and Sarnak's reformulation of the equidistribution theorem of Deligne given in \cite{KS99}. Still, we do not quite succeed in computing the geometric average of $B(C/\mathbb{F}_q)$ but obtain something close to it. This is explained at a later part of this introduction.

At this point, we need to give a definition of the \emph{Linear Independence} property.
\begin{definition}[Linear Independence (LI)]\label{def:LI}
Let $\gamma_1= \sqrt q\mathrm{e}^{\mathrm{i} \theta_1},  \dots, \gamma_{2g}= \sqrt q\mathrm{e}^{\mathrm{i} \theta_{2g}}$ be the inverse zeros of a curve $C/\mathbb{F}_q$.
We say that $C$ satisfies the \emph{Linear Independence} (LI) property if the set, 
\[
\{ \theta_j \mid 0\le \theta_j \le \pi \text{ with }j = 1, \dots, 2g \} \cup \{ \pi \},
\]
counting the inverse zeros with multiplicity, is linearly independent over $\mathbb{Q}$. 
\end{definition}
The number field version of LI, which states that nonnegative ordinates of the critical zeros of the Riemann zeta function or Dirichlet $L$-functions are linearly independent over $\mathbb{Q}$, plays a key role in the aforementioned work \cite{Ng04} of Ng, as well as in the work of Rubinstein and Sarnak on the prime number race in \cite{RS94}. Note that LI is called the \emph{Grand Simplicity Hypothesis} in \cite{RS94} and \cite{Cha08}.
The fact that such a property on zeta zeros has a consequence in understanding the exact growth of $M(x)$ had already been made clear by Ingham in \cite{Ing42}, well before Odlyzko and te Riele disproved Mertens's conjecture. Ingham proved in 1942 that $\limsup_{x\to\infty}M(x)/\sqrt x = \infty$ if LI holds true for the Riemann zeta function.

Unfortunately, there is currently very little direct theoretical or numerical evidence to support the number field version of LI. However, in the function field case, things are better understood. There are known examples (see \cite{Cha08}, \cite{CI11} and \cite{Kow08}) where LI can be confirmed positively and, in some other cases, negatively. Moreover, the work \cite{Kow08} of Kowalski shows that most curves in a certain one-parameter family of hyperelliptic curves satisfy LI. For more background and precise statements of Kowalski's results, the readers are referred to \cite{Kow08} and Remark \ref{rmk:LI} of this paper.

The importance of LI in our work comes from Theorem \ref{thm:GonekConjectureAnalog}, which states that, if $C$ satisfies LI, then the bound 
\[
D(C/\mathbb{F}_q):=
\frac1{q^{1/2}}
\sum_{\gamma}
\left|
\frac{\gamma}{{Z_{C/\mathbb{F}_q}}'({\gamma}^{-1})} \frac{\gamma}{\gamma - 1}
\right|
\]
of $B(C/\mathbb{F}_q)$ we found in Corollary \ref{thm:MertensConjectureFunctionField} becomes sharp, that is, $B(C/\mathbb{F}_q) = D(C/\mathbb{F}_q)$ under LI.
It turns out that Kowalski's argument in \cite{Kow08} can be easily extended to prove that \emph{most} curves in the family $\mathcal{H}_{2g+1}$ of hyperelliptic curves of genus $g$ satisfy LI (Theorem \ref{thm:MostLI}). As a consequence, we have that $B(C/\mathbb{F}_q) = D(C/\mathbb{F}_q)$ for most curves $C/\mathbb{F}_q$ in $\mathcal{H}_{2g+1}$. This is why we choose to work with $\mathcal{F} := \mathcal{H}_{2g+1}$ in this paper.  

The next step is to use Deligne's equidistribution theorem to find the geometric average of $D(C/\mathbb{F}_q)$. 
We show in Theorem \ref{thm:MainTheorem} that a certain \emph{truncated} version of the geometric average of $D(C/\mathbb{F}_q)$ is equal to the integral
\[
\mathcal{I}(g):= \int_{\USp(2g, \mathbb{C})} \varphi(U) \, \mathrm{d}\mu_{\mathrm{Haar}}(U).
\]
Here, $\mathrm{d}\mu_{\mathrm{Haar}}$ is the unique probability Haar measure on the unitary symplectic group $\USp(2g, \mathbb{C})$. The function $\varphi$ is defined in the equations \eqref{eq:ZU} and \eqref{eq:Varphi}, to which we refer the readers for its definition. But, we note here that this situation is similar to the number field case in \cite{Ng04} where the discrete negative moments $J_{-k}(T)$ of $\zeta'(s)$ play an important role in understanding the growth of $M(x)$.

In \S\ref{sec:RMT}, which can be read independently of other parts of this paper, we study the integral $\mathcal{I}(g)$, especially its asymptotic behavior as $g\to\infty$.
A key result here is Theorem \ref{thm:MomentThm}, which finds asymptotic expression of the average value of positive powers of the characteristic polynomials of random unitary symplectic matrices. 
Main tool is a recent formula in \cite{DIK11} by Deift, Its, and Krasovsky on Hankel's determinants with singular weight functions. 

Our result seems to suggest that the geometric average of $B(C/\mathbb{F}_q)$ over $\mathcal{H}_{2g+1}$ is given by the asymptotic formula in \eqref{eq:AnalogOfB}.
However, there are, at least, two major issues we cannot resolve in this paper.
First is the extent of the possible failure of LI in $\mathcal{H}_{2g+1}$. Without LI, it may happen that $B(C/\mathbb{F}_q) < D(C/\mathbb{F}_q)$, therefore, the geometric average of $D(C/\mathbb{F}_q)$ might potentially overestimate that of $B(C/\mathbb{F}_q)$. Even though the set of conjugacy classes in $\USp(2g, \mathbb{C})$ whose eigenvalues have no (nontrivial multiplicative) relations form a measure zero subset with respect to the Haar measure, it is still dense in $\USp(2g, \mathbb{C})$, and it is unclear if one could utilize the equidistribution theorem to control the difference between the averages of $B(C/\mathbb{F}_q)$ and $D(C/\mathbb{F}_q)$.
The second issue is that we cannot (yet) finish the proof of the asymptotic formula \eqref{eq:AnalogOfB} of $\mathcal{I}(g)$. To do so, it seems that we need to have a finer control on the error term in the formula of Deift, Its, Krasovsky. For more detailed explanation on this point, see \S\ref{sec:RMT}, especially, Remark \ref{rmk:Error}.

\subsection*{Acknowledgements} 
The author is greatly indebted to Emmanuel Kowalski for helpful discussion and to Jinho Baik, who brought to the author's attention the work of Deift, Its and Krasovsky \cite{DIK11} and explained its usefulness in proving Theorem \ref{thm:MomentThm}. 
Also, the author is thankful to Peter Humphries, who found an error in Proposition \ref{thm:EquationMX} in an earlier version of this paper, in addition to offering many other helpful remarks.

\section{Asymptotic formula and the Linear Independent property}\label{sec:Asymptotic}
Throughout this paper, we write $\#A$ for the cardinality of for a finite set $A$. We fix a power $q$ of an odd prime $p>2$ and we denote by $\mathbb{F}_q$ a finite field with $q$ elements. For each $n\ge1$, we have a unique extension $\mathbb{F}_{q^n}$ (inside a chosen algebraic closure of $\mathbb{F}_q$) over $\mathbb{F}_q$ of degree $n$. For a nonsingular projective of curve $C$ over $\mathbb{F}_q$ of genus $g$, \emph{the zeta function $Z_{C/\mathbb{F}_q}(u)$ of $C$ over} $\mathbb{F}_q$ is defined by
\[ Z_{C/\mathbb{F}_q}(u) := \exp\left( \sum_{n\ge 1} \frac{\# C(\mathbb{F}_{q^n})}{n} u^n\right),
\]
which is initially defined as a formal power series in $u$ with rational coefficients. It is known from the Riemann hypothesis for curves over finite fields that
\begin{equation}\label{eq:Zeta}
Z_{C/\mathbb{F}_q}(u) = \frac{P_{C/\mathbb{F}_q}(u)}{(1-u)(1-qu)},
\end{equation}
where $P_{C/\mathbb{F}_q}(u)$ is a polynomial in $u$ with integer coefficients of degree $2g$, which factorizes as
\begin{equation}\label{eq:RiemannHypothesis}
P_{C/\mathbb{F}_q}(u) = \prod_{j=1}^{2g}(1-\gamma_ju)
\end{equation}
for some complex numbers $\gamma_j$ with $|\gamma_j| = \sqrt q$ for all $j = 1, \dots, 2g$. These numbers $\gamma_j$ are called the \emph{inverse zeros} of $C$. By an \emph{order} of an inverse zero $\gamma$, we mean the multiplicity of $\gamma^{-1}$ as a root of $P_{C/\mathbb{F}_q}(u)$. If $\gamma$ is of order one, we will say that the inverse zero $\gamma$ is \emph{simple}.

\subsection{Asymptotic formula of $M_{C/\mathbb{F}_q}(X)$ }
Define $Z_{\mu}(u)$ to be the following Dirichlet series (in $u$) associated with $\mu_{C/\mathbb{F}_q}(D)$, together with the change of variable $u:= q^{-s}$,
\begin{equation}\label{eq:DirichletMoebius}
Z_{\mu}(u) :=  \sum_{D\ge 0} \frac{\mu_{C/\mathbb{F}_q}(D)}{\mathcal{N}D^s} =  \sum_{N=0}^{\infty} c_{\mu}(N) u^N.
\end{equation}
Here, $\mathcal{N}D$ is the absolute norm of the divisor $D$ and $c_{\mu}(N) := \sum_{\deg(D) = N} \mu_{C/\mathbb{F}_q}(D)$. From the Euler product expression of $Z_{\mu}(u)$, it is easy to show (see Chapter 1 of \cite{Ros02} or follow the same argument as in the number field case) that 
\begin{equation}\label{eq:DirichletInverse}
Z_{\mu}(u) = \frac1{Z_{C/\mathbb{F}_q}(u)} = \frac{(1-u)(1-qu)}{P_{C/\mathbb{F}_q}(u)}.
\end{equation} 
From \eqref{eq:DirichletMoebius} and the definition of $c_{\mu}(N)$,  we have that 
\begin{equation}\label{eq:MasCmu}
M_{C/\mathbb{F}_q}(X) = \sum_{N\le X} c_{\mu}(N).
\end{equation}
Therefore, the crucial step in finding the asymptotic formula of $M_{C/\mathbb{F}_q}(X)$ is to estimate the coefficients $c_{\mu}(N)$.

First, we consider the easiest case when $C$ is of genus 0. In this case, $P_{C/\mathbb{F}_q}(u) = 1$ and 
\[ Z_{\mu}(u) = (1-u)(1-qu) = 1-(q+1)u +qu^2.\] 
So, $c_{\mu}(N) =0 $ for all $N\ge 3$ and $c_{\mu}(N)=1, -(q+1), q$ if $N=0, 1, 2$ respectively. This easily determines the values of $M_{C/\mathbb{F}_q}(X)$ for all $X$. In particular, we obtain the trivial bound of $|M_{C/\mathbb{F}_q}(X)| \le q$ for all $X$.

Next, we consider the general case for arbitrary genus $g$.
Let $C_1$ is a circular path in the complex plane of radius $1$ centered at the origin, oriented counterclockwise. We calculate the integral 
\[ \frac1{2\pi i}\int_{C_1} \frac{Z_{\mu}(u)}{u^{N+1}}\, \mathrm{d}u \]
using Cauchy's theorem. 

First, we note that the above integral can be easily bounded independently of $N$. From \eqref{eq:DirichletInverse} and \eqref{eq:RiemannHypothesis},
\begin{equation}\label{eq:IntegralBound}
\left|\frac1{2\pi i}\int_{C_1} \frac{Z_{\mu}(u)}{u^{N+1}}\, \mathrm{d}u \right| \le \frac1{2\pi} \int_{C_1} \left| \frac{Z_{\mu}(u)}{u^{N+1}}\right| |\mathrm{d}u| \le  
\frac{2(1+q)}{(1-\sqrt q)^{2g}}.
\end{equation}

Next, we see from \eqref{eq:DirichletMoebius}, \eqref{eq:DirichletInverse}, and \eqref{eq:RiemannHypothesis} that the function $Z_{\mu}(u)/u^{N+1}$ has residues at $u=0$ and $u = {\gamma}^{-1}$ for all inverse zeros $\gamma$. The series expression \eqref{eq:DirichletMoebius} implies that the residue of $Z_{\mu}(u)/u^{N+1}$ at $u=0$ is $c_{\mu}(N)$. Therefore, if we define $R_{C/\mathbb{F}_q}(N, \gamma)$ to be the residue of $Z_{\mu}(u)/u^{N+1}$ at $u={\gamma}^{-1}$ for any inverse zero $\gamma$, Cauchy's theorem yields
\[
\left|
c_{\mu}(N) + \sum_{\gamma} R_{C/\mathbb{F}_q}(N, \gamma)
\right|
\le
\frac{2(1+q)}{(1-\sqrt q)^{2g}}.
\]
By letting $N\to\infty$, we get
\begin{equation}\label{eq:Cmu}
c_{\mu}(N) = -\sum_{\gamma} R_{C/\mathbb{F}_q}(N, \gamma) + O(1).
\end{equation}
So, in order to obtain an asymptotic formula for $M_{C/\mathbb{F}_q}(X)$, we will need to calculate the residues $R_{C/\mathbb{F}_q}(N, \gamma)$. We do this by finding the Laurent series expansion of $Z_{\mu}(u)/u^{N+1} = 1/(Z_{C/\mathbb{F}_q}(u)u^{N+1})$ directly. 

From binomial theorem,
\begin{equation}\label{eq:UN}
u^{-(N+1)} = 
\gamma^{N+1}\sum_{k=0}^{\infty} (-1)^k\binom{N+k}{k} \gamma^k (u- \gamma^{-1})^k. 
\end{equation}
Let $r$ be the order of $\gamma$. Then the power series expansion of $Z_{C/\mathbb{F}_q}(u)$ at $u = \gamma^{-1}$ starts with
\[
Z_{C/\mathbb{F}_q}(u) = 
\frac{{Z_{C/\mathbb{F}_q}}^{(r)}(\gamma^{-1})}{r!}(u - \gamma^{-1})^r + \cdots.
\]
Here ${Z_{C/\mathbb{F}_q}}^{(r)}(u)$ is the $r$-th derivative of $Z_{C/\mathbb{F}_q}(u)$ (with respect to the variable $u$). Then, the Laurent series expansion of $1/Z_{C/\mathbb{F}_q}(u)$ at $u = \gamma^{-1}$ begins with
\begin{equation}\label{eq:Laurent}
\frac1{Z_{C/\mathbb{F}_q}(u)} = 
\frac{r!}{{Z_{C/\mathbb{F}_q}}^{(r)}(\gamma^{-1})}(u - \gamma^{-1})^{-r} + \cdots.
\end{equation}
Therefore, the residue $R_{C/\mathbb{F}_q}(N, \gamma)$ is obtained by multiplying the two series \eqref{eq:UN} and \eqref{eq:Laurent} and extracting the coefficient of $(u - \gamma^{-1})^{-1}$. To be precise,
\begin{align}
R_{C/\mathbb{F}_q}(N, \gamma) &= \frac{r!}{{Z_{C/\mathbb{F}_q}}^{(r)}(\gamma^{-1})} \gamma^{N+1}(-1)^{r-1}  \binom{N+r-1}{r-1}\gamma^{r-1} + \cdots \nonumber \\
&= \frac{r!}{{Z_{C/\mathbb{F}_q}}^{(r)}(\gamma^{-1})} \gamma^{N+1} (-1)^{r-1} \frac{N^{r-1}}{(r-1)!} \gamma^{r-1} + \cdots \nonumber \\
&=
\frac{\gamma^{N+r}(-1)^{r-1} r}{{Z_{C/\mathbb{F}_q}}^{(r)}(\gamma^{-1})} N^{r-1} + \cdots,\label{eq:RN}
\end{align}
where all the suppressed terms are polynomials in $N$ of degree $r-2$ or less. 

If we sum the equation \eqref{eq:RN} over $N=1, \cdots, X$, we get
\begin{equation}\label{eq:Asymptotic}
\sum_{N=1}^{X} R_{C/\mathbb{F}_q}(N, \gamma) =
\frac{\gamma^{r}(-1)^{r-1} r}{{Z_{C/\mathbb{F}_q}}^{(r)}(\gamma^{-1})}
\frac{\gamma}{\gamma-1}
X^{r-1} \gamma^X
+ 
O(X^{r-2} \gamma^X),
\end{equation}
as $X\to\infty$. This is an immediate consequence of Lemma \ref{thm:Series}, which can be proved by partial summation \cite[Theorem 4.2]{Apo76}, as outlined in \cite[Lemma 2.2]{Cha08}. So, we omit the proof.
\begin{lemma}\label{thm:Series}
Let $\beta$ be a complex number with $|\beta| > 1$ and $k$ be a nonnegative integer. 
\[
\lim_{X\to\infty} \frac1{X^k \beta^X}
\left(
\sum_{N = 1}^X N^k \beta^N
\right)
=
\frac{\beta}{\beta-1}.
\]
\end{lemma}

Denote by $\theta(\gamma)$ the argument of the complex number $\gamma$, so that $\gamma = \sqrt q\mathrm{e}^{\mathrm{i} \theta(\gamma)}$. Then, \eqref{eq:Asymptotic} becomes
\begin{equation}\label{eq:AsymptoticNormalized}
-\frac1{X^{r-1} q^{X/2}} \sum_{N=1}^{X} R_{C/\mathbb{F}_q}(N, \gamma) =
\frac{(-\gamma)^{r} r}{{Z_{C/\mathbb{F}_q}}^{(r)}(\gamma^{-1})}
\frac{\gamma}{\gamma-1}
\mathrm{e}^{\mathrm{i} X\theta(\gamma)} 
+ 
o(1).
\end{equation}
Now, the equations \eqref{eq:MasCmu}, \eqref{eq:Cmu}, \eqref{eq:AsymptoticNormalized} together yield the estimate of $M_{C/\mathbb{F}_q}(X)$ in Proposition \ref{thm:EquationMX}.

\begin{proposition}\label{thm:EquationMX}
For an inverse zero $\gamma$, let $\theta(\gamma)$ be the argument of $\gamma$, so that $\gamma = \sqrt q \mathrm{e}^{\mathrm{i}\theta(\gamma)}$. Also, let $r$ be the maximum order among all the inverse zeros $\gamma$ of $Z_{C/\mathbb{F}_q}(u)$, that is,
\[
r = \max\{ \text{order of } \gamma_j\}_{j=1}^{2g}
\]
Then, as $X\to\infty$,
\[ 
\frac{M_{C/\mathbb{F}_q}(X)}{X^{r-1}q^{X/2}} = 
\sum_{\mathrm{ord}(\gamma)=r} \frac{(-\gamma)^r r}{{Z_{C/\mathbb{F}_q}}^{(r)}({\gamma}^{-1})} \frac{\gamma}{\gamma - 1} \mathrm{e}^{\mathrm{i} X\theta(\gamma)}+ o(1).
\]
In particular, if all inverse zeros $\{ \gamma_j\}_{j=1}^{2g}$ are simple, then
\[ 
\frac{M_{C/\mathbb{F}_q}(X)}{q^{X/2}} = 
-\sum_{j=1}^{2g} \frac{\gamma_j}{{Z_{C/\mathbb{F}_q}}'({\gamma_j}^{-1})} \frac{\gamma_j}{\gamma_j - 1} \mathrm{e}^{\mathrm{i} X\theta(\gamma_j)}+ o(1).
\]
\end{proposition}

\begin{corollary}\label{thm:MertensConjectureFunctionField}
The notations are as above. 
\[ 
\limsup_{X\to\infty}\frac{M_{C/\mathbb{F}_q}(X)}{X^{r-1}q^{X/2}} 
\le
\sum_{\mathrm{ord}{\gamma}=r}
\left|
\frac{\gamma^r r}{{Z_{C/\mathbb{F}_q}}^{(r)}({\gamma}^{-1})} \frac{\gamma}{\gamma - 1}
\right|.
\]
\end{corollary}

The next theorem follows from the adaptation of Rubinstein and Sarnak's argument in \cite{RS94} for the function field setting and is identical to that of Theorem 3.2 in \cite{Cha08}, so we omit its proof.
\begin{theorem}\label{thm:LimitingDistribution}
The notations are the same as in Proposition \ref{thm:EquationMX}.
The function $M_{C/\mathbb{F}_q}(X)/(X^{r-1}q^{X/2})$ has a limiting distribution  $\mu$ on $\mathbb{R}$, that is,
\[ \lim_{Y\to\infty}\frac1Y\sum_{X=1}^Y 
f\left( \frac{M_{C/\mathbb{F}_q}(X)}{X^{r-1}q^{X/2}} \right) = \int_{-\infty}^{\infty} f(x) \, \mathrm{d}\mu(x)
\]
for all bounded continuous function $f$ on $\mathbb{R}$.
\end{theorem}

\subsection{Application of the Linear Independence property}
Suppose that the curve $C$ satisfies the Linear Independence property (Definition \ref{def:LI}). One immediate consequence of LI is that all of inverse zeros of $C$ are simple, therefore, the formula in Proposition \ref{thm:EquationMX} becomes
\begin{equation}\label{eq:EquationMXWithLI}
\frac{M_{C/\mathbb{F}_q}(X)}{q^{X/2}} = -\sum_{j=1}^{2g} \left|\frac{\gamma_j}{{Z_{C/\mathbb{F}_q}}'({\gamma_j}^{-1})} \frac{\gamma_j}{\gamma_j - 1}\right| \cos(\omega(\gamma_j) + X\theta(\gamma_j))+ o(1), 
\end{equation}
where $\omega(\gamma_j)$ is the argument
\[ \omega(\gamma_j) := \arg\left( \frac{{\gamma_j}}{{Z_{C/\mathbb{F}_q}}'({\gamma_j}^{-1})} \frac{\gamma_j}{\gamma_j - 1} \right).\]

Another consequence of LI is that the measure $\mu$ whose existence is stated in Theorem \ref{thm:LimitingDistribution} is absolutely continuous and we can write down the Fourier transform $\hat{\mu}$ explicitly
\begin{equation}\label{eq:FourierTransform}
\hat{\mu}(\xi) = \prod_{j=1}^{g} J_0\left(2\left|
\frac{{\gamma_j}}{{Z_{C/\mathbb{F}_q}}'({\gamma_j}^{-1})} \frac{\gamma_j}{\gamma_j - 1}\right|  \xi
\right),
\end{equation}
where $J_0(z)$ is the Bessel function of the first kind. 
Again, (the adaptation of) Rubinstein and Sarnak's proof can be used to establish this, so we omit the proof of this formula. However, this fact allows us to prove that the sum of magnitudes of the oscillating terms in the right side of \eqref{eq:EquationMXWithLI}, which provides a bound of the left side of \eqref{eq:EquationMXWithLI}, does indeed give the sharp bound as $X\to\infty$. This is the key observation which we will use in the next sections, and we summarize it as a theorem below.

\begin{theorem}\label{thm:GonekConjectureAnalog}
If $C$ satisfies LI, then we have
\[
\limsup_{X\to \infty} \frac{M_{C/\mathbb{F}_q}(X)}{q^{X/2}} 
=
\sum_{j=1}^{2g} \left|\frac{{\gamma_j}}{{Z_{C/\mathbb{F}_q}}'({\gamma_j}^{-1})} \frac{\gamma_j}{\gamma_j - 1}\right|.
\]
\end{theorem}

\section{Universal families of hyperelliptic curves}\label{sec:HyperellipticFamily}
We define $\mathcal{H}_{2g+1}$ to be the space of monic polynomials of degree $2g+1$ with distinct roots (see (10.1.18.1) of \cite{KS99}). One can think of $\mathcal{H}_{2g+1}$ as an open subvariety of the affine scheme $\mathbb{A}^{2g+1}$ over $\mathbb{Z}$. In particular, for each $n\ge1$, $\mathcal{H}_{2g+1}(\mathbb{F}_{q^n})$ is the set of monic polynomials $f(x)=a_0 + a_1x + \cdots + a_{2g}x^{2g} + x^{2g+1}$ with coefficients $a_i$ in $\mathbb{F}_{q^n}$ where the discriminant of this polynomial is nonzero. Therefore, each $f\in \mathcal{H}_{2g+1}(\mathbb{F}_{q^n})$ defines a hyperelliptic curve $C_f$ of genus $g$ over $\mathbb{F}_{q^n}$, the nonsingular projective model of the plane curve defined by the equation $y^2 = f(x)$. At this point, it will be convenient to introduce a terminology from \cite{Cha97}. We will say that \emph{most} points of $\mathcal{H}_{2g+1}$ have the property $D= \{D_n\}_{n=1}^{\infty}$ if
\[
\lim_{n\to\infty} 
\frac{\#\{ f\in \mathcal{H}_{2g+1}(\mathbb{F}_{q^n}) \mid 
C_f \text{ satisfies }D_n\} }{\#\mathcal{H}_{2g+1}(\mathbb{F}_{q^n})} 
=1.
\]

\subsection{LI for most curves in $\mathcal{H}_{2g+1}$}
\begin{theorem}[Chavdarov \cite{Cha97}, Kowalski \cite{Kow08}]\label{thm:MostLI} For fixed $q$ and $g$, we have that
\[
\lim_{n\to\infty} 
\frac{\#\{ f\in \mathcal{H}_{2g+1}(\mathbb{F}_{q^n}) \mid C_f \text{ satisfies $\mathrm{LI}$. }\} }{\#\mathcal{H}_{2g+1}(\mathbb{F}_{q^n})} = 1.
\] 
In other words, most points of $\mathcal{H}_{2g+1}$ satisfy $\mathrm{LI}$.
\end{theorem}
\begin{proof}
Let $f$ and $C_f$ be as above. Then one can show that, for most points of $\mathcal{H}_{2g+1}$, the sum of inverse zeros of $C_f$ is nonzero. This directly follows from Deligne's equidistribution theorem (Theorem 10.8.2 of \cite{KS99}) because the set of conjugacy classes with zero trace forms a measure zero subset of the space of conjugacy classes $\USp(2g, \mathbb{C})^{\#}$ of $\USp(2g, \mathbb{C})$, with respect to the (direct image of) Haar measure. 

The second step is to apply Chavdarov's theorem (Theorem 2.3 of \cite{Cha97}) which says that, for most points of $\mathcal{H}_{2g+1}$, the Galois group of the splitting field of $P_{C_f}(u)$ is as large as possible, that is, the Galois group is isomorphic to the Weyl group $W_{2g}$ corresponding to the symplectic group $\Sp(2g)$. To apply Chavdarov's theorem, we need to ensure that the mod-$\ell$ geometric monodromy group of $\mathcal{H}_{2g+1}$ is $\Sp(2g, \mathbb{F}_{\ell})$ for all large $\ell$. But this result had been previously obtained by J.~K.~Yu, which is unpublished. More recently, Hall in \cite{Hal08} and, independently, Achter and Pries in \cite{AP07}, proved this result. 

The last step is now to follow Kowalski's argument in \S3 of \cite{Kow08}. 
His proof of Proposition 1.1 in \S3 of \cite{Kow08} can be applied to $\mathcal{H}_{2g+1}$ without any change to show that, if the sum of inverse zeros of $C$ is nonzero and the Galois group of $P_{C_f}(u)$ is as large as possible, then $C$ satisfies LI. This concludes the proof of our theorem that most elements of $\mathcal{H}_{2g+1}$ satisfy LI.
\cite{Cha97}
\end{proof}
\begin{remark}\label{rmk:LI}
The aforementioned Chavdarov's theorem has a quantitatively refined version, which was first proved by Kowalski. See Theorems 6.1 and 6.2 in \cite{Kow06b}, where Kowalski gives a quantitative bound on the number of curves in a family whose zeta functions are either reducible or have splitting fields with strictly smaller Galois groups than the maximum possible one. And, using this result, Kowalski derives a bound of the number of curves which don't satisfy LI in the following one-parameter family of hyperelliptic curves
\[ C_t: y^2 = f(x)(x-t),\]
where $f(x)$ is a monic irreducible polynomial with coefficients in $\mathbb{Z}$ whose discriminant is not divisible by $p$ (Proposition 1.1 of \cite{Kow08}).
In fact, if we assume that $p>2g+1$, then Theorem 6.1 (ii) of \cite{Kow06b} is directly applicable to the family $\mathcal{H}_{2g+1}$, and we can deduce from it that the number $N(\mathcal{H}_{2g+1}(\mathbb{F}_q))$ of curves in $\mathcal{H}_{2g+1}(\mathbb{F}_q)$ such that $P_{C_f}(u)$ is either reducible or has splitting field smaller than $W_{2g}$ satisfies, as $q\to\infty$,
\[
N(\mathcal{H}_{2g+1}(\mathbb{F}_q)) \ll q^{2g - \gamma}(\log q),
\]
for $\gamma := 1/(10g^2+6g+8)$. 
Therefore, the same bound above applies to the number of curves that don't satisfy LI. 
\end{remark}

\subsection{Average over the family}
Let $C$ be a nonsingular projective curve over a finite field $\mathbb{F}$ of characteristic $p>2$. As in \S\ref{sec:Asymptotic}, let $r$ be the maximum order of all inverse zeros of $C/\mathbb{F}$. 
Define
\[ B(C/\mathbb{F}) := \limsup_{X\to\infty} \frac{M_{C/\mathbb{F}}(X)}{{\#\mathbb{F}}^{(X+1)/2}X^{r-1}}\]
Further, we let 
\begin{equation}\label{eq:DefinitionD}
D(C/\mathbb{F}) := \frac1{(\#\mathbb{F})^{1/2}} \sum_{\mathrm{ord}(\gamma) = r}
\left|
\frac{\gamma^r r}{{Z_{C/\mathbb{F}}}^{(r)}(\gamma^{-1})}
\frac{\gamma}{\gamma - 1}
\right|.
\end{equation}
Then, Corollary \ref{thm:MertensConjectureFunctionField} and Theorem \ref{thm:GonekConjectureAnalog} can be summarized by saying that 
\begin{equation}\label{eq:Comparison}
B(C/\mathbb{F}) \le D(C/\mathbb{F}),
\end{equation}
and that the equality holds true if $C$ satisfies LI. In this subsection, we investigate a relationship between an average value of $D(C_f, \mathbb{F}_{q^n})$ for all $f \in \mathcal{H}_{2g+1}(\mathbb{F}_{2g+1})$ and a certain integral over the unitary symplectic group $\USp(2g, \mathbb{C})$. 

To describe this relationship, we start by defining the characteristic polynomial $\mathcal{Z}_U(\theta)$. 
(A typographical note: in the literature, this characteristic polynomial is denoted by $Z_U(\theta)$ or $Z(U, \theta)$. But, we use a calligraphic font in this paper to distinguish it from the zeta function $Z_{C/\mathbb{F}_q}(u)$ of a curve $C$.) 
Let $N$ be a positive integer. 
For a $2N \times 2N$ unitary matrix $U$ and a real number $\theta$, we define the function $\mathcal{Z}_U(\theta)$ by
\begin{equation}\label{eq:ZU}
\mathcal{Z}_U(\theta) := \det(I- U\mathrm{e}^{-\mathrm{i}\theta}) = \prod_{m=1}^{2N} (1-\mathrm{e}^{\mathrm{i}(\theta_m-\theta)}),
\end{equation}
where $\mathrm{e}^{\mathrm{i} \theta_1}, \dots, \mathrm{e}^{\mathrm{i}\theta_{2N}}$ are the eigenvalues of $U$. When $U$ has no repeated eigenvalues, we define
\begin{equation}\label{eq:Varphi}
\varphi(U):= \sum_{j = 1}^{2N}\frac1{|{\mathcal{Z}_U}'(\theta_j)|}.
\end{equation}
Note that $\varphi(U)$ depends only on the conjugacy class of $U$. In this paper, we will be mostly interested in $\varphi(U)$ for a unitary symplectic matrix $U\in \USp(2g, \mathbb{C})$. Then, the function $\varphi(U)$ is continuous and well-defined outside the measure zero subset where $U$ has a repeated eigenvalue.

Next, for a curve $C_f$ over $\mathbb{F}_{q^n}$, we recall that there exists a conjugacy class $\vartheta(C_f/\mathbb{F}_{q^n})$ in the set of all conjugacy classes $\USp(2g, \mathbb{C})^{\#}$ of $\USp(2g, \mathbb{C})$. This is called \emph{the unitarized Frobenius conjugacy class attached to} $C_f/\mathbb{F}_{q^n}$. For its definition, the readers are referred to Chapters 9 and 10 (especially \S9.2 and \S\S10.7.2) of \cite{KS99}. In this paper, it will be sufficient to say that this is the unique conjugacy class with the property that
\begin{equation}\label{eq:UnitarizedFrobenius}
P_{C_f/\mathbb{F}_{q^n}}(u) =
\det(1 - u\,q^{n/2}\vartheta(C_f/\mathbb{F}_{q^n})).
\end{equation}

Finally, we define a \emph{truncated version} of $D(C_f, \mathbb{F}_{q^n})$ using $\varphi(U)$ above. Fix a positive number $T>0$. Then, 
\begin{equation}\label{eq:TruncatedD}
D^T(C_f/\mathbb{F}_{q^n}) := 
\begin{cases}
D(C_f, \mathbb{F}_{q^n}) & \text{ if }\varphi(\vartheta(C_f/\mathbb{F}_{q^n})) \le T,\\
0 &  \text{otherwise.} 
\end{cases}
\end{equation}
Note that the second case in the above definition is used when either $\varphi(\vartheta(C_f/\mathbb{F}_{q^n})) > T$ or $\vartheta(C_f/\mathbb{F}_{q^n})$ has a repeated eigenvalue.
The \emph{truncated average} over $\mathcal{H}_{2g+1}(\mathbb{F}_{q^n})$ is defined by
\begin{equation}\label{eq:TruncatedAverageD}
\overline{D}^T(\mathcal{H}_{2g+1}(\mathbb{F}_{q^n}))
:= 
\frac1{\#\mathcal{H}_{2g+1}(\mathbb{F}_{q^n})}
\sum_{f\in \mathcal{H}_{2g+1}(\mathbb{F}_{q^n})}
D^T(C_f/\mathbb{F}_{q^n}) .
\end{equation}
We are ready to state the main theorem of this section.
\begin{theorem}\label{thm:MainTheorem}
Notations are as above.
\[
\lim_{T\to\infty}
\lim_{n\to\infty}
\overline{D}^T(\mathcal{H}_{2g+1}(\mathbb{F}_{q^n}))
=
\int_{\USp(2g, \mathbb{C})} \varphi(U) \, \mathrm{d}\mu_{\mathrm{Haar}}(U).
\]
Here, $\mathrm{d}\mu_{\mathrm{Haar}}$ is the unique probability Haar measure on $\USp(2g, \mathbb{C})$.
\end{theorem}

Using \eqref{eq:ZU} and \eqref{eq:UnitarizedFrobenius}, we can easily deduce that
\begin{equation}\label{eq:ZandP}
\mathcal{Z}_{\vartheta(C_f/\mathbb{F}_{q^n})}(\theta) = P_{C_f/\mathbb{F}_{q^n}}((q^{n/2} \mathrm{e}^{\mathrm{i} \theta})^{-1}),
\end{equation}
for any real $\theta$.
Assume that $C_f$ has only simple inverse zeros and write 
$\gamma_j = q^{n/2}\mathrm{e}^{\mathrm{i} \theta_j}$.
Then, we differentiate the equation \eqref{eq:ZandP} to obtain
\begin{equation}\label{eq:Rewrite}
\frac{\gamma_j}{{Z_{C_f/\mathbb{F}_{q^n}}}'({\gamma_j}^{-1})} \frac{\gamma_j}{\gamma_j - 1} =  \frac{1 - (q^n/\gamma_j)}{\mathrm{i} \,{\mathcal{Z}_{\vartheta(C_f/\mathbb{F}_{q^n})}}'(\theta_j)}.
\end{equation}
Further, assume that $\varphi(\vartheta(C_f/\mathbb{F}_{q^n}))\le T$. Then, we sum \eqref{eq:Rewrite} over $j = 1, \dots, 2g$, and, by setting $r=1$ in \eqref{eq:DefinitionD}, this yields 
\begin{align}\label{eq:ResultD}
D(C_f/\mathbb{F}_{q^n}) 
&=
\frac1{q^{n/2}}
\sum_{j=1}^{2g}
\left|
\frac{1 - (q^n/\gamma_j)}{\mathrm{i} \,{\mathcal{Z}_{\vartheta(C_f/\mathbb{F}_{q^n})}}'(\theta_j)}
\right|
 \notag \\
&=
\sum_{j=1}^{2g}
\left(
\frac1{|{\mathcal{Z}_{\vartheta(C_f/\mathbb{F}_{q^n})}}'(\theta_j)|}
+
\frac1{q^{n/2}}
\frac{A(f, j)}{|{\mathcal{Z}_{\vartheta(C_f/\mathbb{F}_{q^n})}}'(\theta_j)|}
\right), 
\end{align}
where $A(f,j)$ is a constant with $|A(f,j)|\le 1$.
Now, we compute $\overline{D}^T(\mathcal{H}_{2g+1}(\mathbb{F}_{q^n})$ from its definition \eqref{eq:TruncatedAverageD} by adding \eqref{eq:ResultD} over all $f\in \mathcal{H}_{2g+1}$ for those $f$ with
$\varphi(\vartheta(C_f/\mathbb{F}_{q^n}))\le T$. 
As a result,
\begin{multline}\label{eq:PassToLimit}
\overline{D}^T(\mathcal{H}_{2g+1}(\mathbb{F}_{q^n})
=
\frac1{\#\mathcal{H}_{2g+1}(\mathbb{F}_{q^n})}
\sum_{\varphi(\vartheta(C_f/\mathbb{F}_{q^n})) \le T}
\sum_{j=1}^{2g}
\frac1{|{\mathcal{Z}_{\vartheta(C_f/\mathbb{F}_{q^n})}}'(\theta_j)|}
+
\\
\frac1{\#\mathcal{H}_{2g+1}(\mathbb{F}_{q^n})}
\sum_{\varphi(\vartheta(C_f/\mathbb{F}_{q^n})) \le T}
\sum_{j=1}^{2g}
O\left(
\frac1{q^{n/2}}
\frac1{|{\mathcal{Z}_{\vartheta(C_f/\mathbb{F}_{q^n})}}'(\theta_j)|}
\right).
\end{multline}

The next step is to pass $n\to\infty$ in \eqref{eq:PassToLimit} and to apply Deligne's equidistribution theorem, Theorem 10.8.2 of \cite{KS99}. The right side of the first line of \eqref{eq:PassToLimit} then becomes
\[
\lim_{n\to\infty}
\frac1{\#\mathcal{H}_{2g+1}(\mathbb{F}_{q^n})}
\sum_{\varphi(\vartheta(C_f/\mathbb{F}_{q^n})) \le T}
\sum_{j=1}^{2g}
\frac1{|{\mathcal{Z}_{\vartheta(C_f/\mathbb{F}_{q^n})}}'(\theta_j)|}
=
\int_{\varphi \le T}\varphi(U) \, \mathrm{d}\mu_{\mathrm{Haar}}(U).
\]
The second line of \eqref{eq:PassToLimit} converges to zero as $n\to\infty$ due to the $q^{n/2}$ term in the denominator and the convergence of the first line of \eqref{eq:PassToLimit}. In other words, we proved that
\[
\lim_{n\to\infty}
\overline{D}^T(\mathcal{H}_{2g+1}(\mathbb{F}_{q^n}))
=
\int_{\varphi\le T} \varphi(U) \, \mathrm{d}\mu_{\mathrm{Haar}}(U).
\]
The proof of Theorem \ref{thm:MainTheorem} is now completed by letting $T\to\infty$.

\section{Averages of characteristic polynomials on unitary symplectic groups}\label{sec:RMT}
Recall that, for a $2N \times 2N$ unitary matrix $U$ and a real number $\theta$, the function $\mathcal{Z}_U(\theta)$ was defined in \eqref{eq:ZU} by
\[
\mathcal{Z}_U(\theta) := \det(I- U\mathrm{e}^{-\mathrm{i}\theta}) = \prod_{m=1}^{2N} (1-\mathrm{e}^{\mathrm{i}(\theta_m-\theta)}),
\]
where $\mathrm{e}^{\mathrm{i} \theta_1}, \dots, \mathrm{e}^{\mathrm{i}\theta_{2N}}$ are the eigenvalues of $U$. Also, the function $\varphi(U)$ is defined in \eqref{eq:Varphi} by
\[
\varphi(U):= \sum_{j = 1}^{2N}\frac1{|{\mathcal{Z}_U}'(\theta_j)|},
\]
whenever $U$ has no repeated eigenvalues. When $U$ is an element in $\USp(2N, \mathbb{C})$, its eigenangles $\theta_1, \dots, \theta_{2N}$  come in complex conjugate pairs. And we will enumerate them in the way that $0\le \theta_j \le \pi$ for $j = 1, \dots, N$ and $\theta_{N+1} = -\theta_1, \dots, \theta_{2N} = -\theta_N$. 
The main theorem of this section is Theorem \ref{thm:MomentThm}, where we give an asymptotic formula of the $2s$-th moments, for any positive real number $s$, of $\mathcal{Z}_U(\theta)$ in $\USp(2N, \mathbb{C})$ using a recent work \cite{DIK11} of Deift, Its, and Krasovsky. 

\begin{theorem}[cf.~Theorem 5 in \cite{KO08}]\label{thm:MomentThm}
Fix a real number $s>0$ (not necessarily an integer) and $\theta$ with $0< \theta <\pi$. As $N\to\infty$,
\[
\int_{\USp(2N, \mathbb{C})} |\mathcal{Z}_U(\theta)|^{2s} \, \mathrm{d}\mu_{\mathrm{Haar}}(U)
\sim
N^{(s^2)}2^{-s} (\sin\theta)^{-s(s+1)}\frac{G(1+s)^2}{G(1+2s)}.
\]
Here, $G(z)$ is the Barnes $G$-function.
\end{theorem}
\begin{proof}
During the proof, we use the notation
\begin{equation}\label{eq:DefTwoKMoment}
\langle |\mathcal{Z}_U(\theta)|^{2s} \, \rangle_{\USp(2N, \mathbb{C})}:= 
\int_{\USp(2N, \mathbb{C})} |\mathcal{Z}_U(\theta)|^{2s} \, \mathrm{d}\mu_{\mathrm{Haar}}(U).
\end{equation}
We first rewrite 
\begin{equation}\label{eq:TrigIdentity2}
|\mathcal{Z}_U(\theta)| = 
\prod_{j = 1}^{N}|1 - \mathrm{e}^{\mathrm{i}(\theta_j - \theta)}||1 - \mathrm{e}^{\mathrm{i}(\theta_j + \theta)}|
=2^N \prod_{j=1}^N |\cos\theta_j - \cos\theta|.
\end{equation}
This can be done by applying the following straightforward trigonometric identity
\begin{equation}\label{eq:TrigIdentity}
\left|1 - \mathrm{e}^{\mathrm{i}(\theta_j - \theta_k)}\right|\left|1 - \mathrm{e}^{\mathrm{i}(\theta_j + \theta_k)}\right| = 2|\cos\theta_j - \cos\theta_k|,
\end{equation}
to \eqref{eq:ZU}. To integrate \eqref{eq:TrigIdentity2}, we use the Weyl integration formula, which describes the Haar measures on classical matrix groups explicitly in terms of eigenangles. The version we use here is (5.0.4) of \cite{KS99}, which we recall below. Define a measure $\mu(\USp(2N))$ on $[0, \pi]^N$ by
\begin{equation}\label{eq:HaarOnUSp}
\mathrm{d}\mu(\USp(2N)) := \frac{2^{N^2}}{N!\pi^N} \prod_{1\le j<k \le N}(\cos\theta_j - \cos\theta_k)^2 \prod_{j=1}^N \sin^2\theta_j \prod_{j=1}^N \mathrm{d}\theta_j,
\end{equation}
where $\mathrm{d}\theta_1, \dots, \mathrm{d}\theta_N$ are the usual Lebesgue measure on the set $[0, \pi]$. Then, the Weyl integration formula says that, for a bounded, Borel measurable $\mathbb{R}$-valued central functions $g$ on $\USp(2N, \mathbb{C})$, we have
\begin{equation}\label{eq:WeylIntegration}
\int_{\USp(2N, \mathbb{C})} g(U)\, \mathrm{d}\mu_{\mathrm{Haar}}(U) = \int_{[0, \pi]^N} \tilde{g}(\theta_1, \dots, \theta_N) \, \mathrm{d}\mu(\USp(2N)).
\end{equation}
Here, $\tilde{g}$ is the function on $[0, \pi]^N$ defined by the property
\[ \tilde{g}(\theta_1, \dots, \theta_N) = g(U),\]
whenever $\theta_1, \dots, \theta_N, -\theta_1, \dots, -\theta_N$ are the eigenangles of $U\in \USp(2N, \mathbb{C})$. Now, from \eqref{eq:DefTwoKMoment}, \eqref{eq:TrigIdentity2}, \eqref{eq:HaarOnUSp} and \eqref{eq:WeylIntegration}, 
we have
\begin{multline*}
\langle|\mathcal{Z}_U(\theta)|^{2s}\,\rangle_{\USp(2N, \mathbb{C})}  =
\frac{2^{N^2+2sN}}{N!\pi^N} \int_{[0,\pi]^N} 
\prod_{1\le j<k \le N}(\cos\theta_j - \cos\theta_k)^2 \\
\times
\prod_{j=1}^N |\cos\theta_j-\cos\theta|^{2s} \sin^2\theta_j \, \mathrm{d}\theta_j.
\end{multline*}
Note that the expression $\prod_{1\le j<k \le N}(\cos\theta_k - \cos\theta_j)$ is the same as the Vandermonde determinant $\Delta(\cos\theta_1, \dots, \cos\theta_j)$ where 
\[
\Delta(x_1, \dots, x_N):= \det({x_i}^{j-1})_{1\le i, j\le N}.
\]
Set $y=\cos\theta$. Then, clearly, $0<y<1$. Also, we use the change of variables $x_j = \cos\theta_j$ to obtain
\begin{equation}\label{eq:Integral1}
\langle|\mathcal{Z}_U(\theta)|^{2s}\,\rangle_{\USp(2N, \mathbb{C})}  =
\frac{2^{N^2+2sN}}{N!\pi^N} \int_{[-1,1]^N} 
\Delta(x_1, \dots, x_N)^2 \, \prod_{j=1}^N w_y(x_j)\, \mathrm{d} x_j.
\end{equation}
Here, the \emph{weight function} $w_y(x)$ is defined by
\[
w_y(x):= |x-y|^{2s}\sqrt{1-x^2}.
\]
Now, we will use the following \emph{Andr\'eief's identity}, 
\begin{multline}\label{eq:CauchyBinet}
\frac1{N!}\int_{X^n}
\det[f_j(x_k)]_{1\le j, k\le N}\det[g_j(x_k)]_{1\le j, k\le N}
\prod_{j=1}^N w(x_j)\, \mathrm{d} x_j
=\\
\det\left[\int_X f_j(x)g_k(x) \, w(x)\, \mathrm{d} x\right]_{1\le j, k\le N},
\end{multline}
for any interval $X$ in $\mathbb{R}$. Setting $f_j(x) = g_j(x) = x^{j-1}$ and $X=[-1, 1]$ in the Andr\'eief's identity, we can rewrite the integral \eqref{eq:Integral1} as
\begin{equation}\label{eq:Integral2}
\langle|\mathcal{Z}_U(\theta)|^{2s}\,\rangle_{\USp(2N, \mathbb{C})}  =
\frac{2^{N^2+2sN}}{\pi^N} 
\det\left[
\int_{-1}^1 x^{j+k} w_y(x) \, \mathrm{d} x
\right]_{0 \le j, k \le N-1}.
\end{equation}
The determinant in \eqref{eq:Integral2} is, so-called, a determinant of Hankel's type \cite{Sze75} with weight $w_y(x)$. And what we need is an asymptotic expression of Hankel's determinant when the weight function is not differentiable. We use a recent result in \cite{DIK11} of Deift, Its, and Krasovsky, which can be applied to much more general weight functions than our $w_y(x)$. In particular, Theorem 1.20 of \cite{DIK11}, with the following parameters, 
\[
V\equiv 0, \quad b_{\pm}\equiv 1, \quad 
\begin{cases}
\alpha_0 = 1/4, \\
\alpha_1 = s, \\
\alpha_2 = 1/4,
\end{cases} \quad
\begin{cases}
\lambda_0 = 1, \\
\lambda_1 = y, \\
\lambda_2 = -1,
\end{cases}
\quad \text{and} \quad \beta_i = 0 \text{ for all $i$,}
\]
gives 
\begin{multline}\label{eq:Deift}
\det\left[
\int_{-1}^1 x^{j+k} w_y(x) \, \mathrm{d} x
\right]_{0 \le j, k \le N-1} =
4^{-(sN + \frac{N}2 +\frac{s}2 + \frac3{16})} 
(2\pi)^{1/2} N^{s^2 +\frac{1}{4}} \\
\times
2^{-1/8}|1-y^2|^{-s/2}\, G({\textstyle{\frac32}})^{-2}
(1-y^2)^{-s^2/2} \frac{G(1+s)^2}{G(1+2s)} \frac{\pi^{N+\frac12}G({\textstyle{\frac12}})^2}{2^{N(N-1)}N^{1/4}}
(1+o(1)),
\end{multline}
as $N\to\infty$. Further simplification yields
\begin{multline*}
\det\left[
\int_{-1}^1 x^{j+k} w_y(x) \, \mathrm{d} x
\right]_{0 \le j, k \le N-1} =\\
\frac{\pi^N}{2^{N^2 + 2sN}}
2^{-s}N^{s^2}(1-y^2)^{-(s^2+s)/2} \frac{G(1+s)^2}{G(1+2s)}(1+o(1)).
\end{multline*}
Combining this with \eqref{eq:Integral2} (and remembering $y=\cos\theta$), we finish the proof of the theorem.
\end{proof}

In view of Theorem \ref{thm:MainTheorem}, it would be desirable to obtain an asymptotic expression of 
\begin{equation}\label{eq:IG}
\mathcal{I}(N):= \int_{\USp(2N, \mathbb{C})} \varphi(U) \, \mathrm{d}\mu_{\mathrm{Haar}}(U).
\end{equation}
as $N\to\infty$, because this would be then thought of as a function field analog (for the family $\mathcal{H}_{2g+1}$) of the $(\log\log\log x)^{5/4}$ term in \eqref{eq:GonekConjecture}. Using Theorem \ref{thm:MomentThm} (for the case $s =1/2$), we present some evidence in support of the formula
\begin{equation}\label{eq:AnalogOfB}
\mathcal{I}(N) \sim
\sqrt2 G({\textstyle{\frac12}})^2B({\textstyle{\frac58, \frac12}})N^{1/4}.
\end{equation}
Here, $G(z)$ is the Barnes $G$-function and $B(x,y)$ is the beta function
\[ 
B(x, y) = \int_0^1t^{x-1}(1-t)^{y-1}\, \mathrm{d} t.
\]
The rest of the paper is now devoted to presenting the argument in support of the formula \eqref{eq:AnalogOfB}. 
Our computation closely follows the strategy used by Hughes, Keating, and O'Connell in the proof of Theorem 1.2 in \cite{HKO00}.

A straightforward differentiation of $\mathcal{Z}_U(\theta)$ in \eqref{eq:ZU} gives
\begin{equation}\label{eq:ZU1}
|{\mathcal{Z}_U}'(\theta_j)| = |1 - \mathrm{e}^{2\mathrm{i}\theta_j}| \, \prod_{\substack{k = 1\\ k\neq j}}^{N}|1 - \mathrm{e}^{\mathrm{i}(\theta_j - \theta_k)}||1 - \mathrm{e}^{\mathrm{i}(\theta_j + \theta_k)}|,
\end{equation}
for $j = 1, \dots, N$.
Using \eqref{eq:TrigIdentity} and another easy trigonometric identity
\[
|1-\mathrm{e}^{2\mathrm{i}\theta_j}| = 2|\sin\theta_j|,
\]
one easily deduces from \eqref{eq:ZU1} that
\[
|{\mathcal{Z}_U}'(\theta_j)|  = 2^N|\sin\theta_j| \prod_{\substack{k = 1\\ k\neq j}}^{N}|\cos\theta_j - \cos\theta_k|.
\]
Also, obviously, $|{\mathcal{Z}_U}(\theta_{N+j})|=|{\mathcal{Z}_U}(-\theta_j)| = |{\mathcal{Z}_U}(\theta_j)|$ for $j = 1, \dots, N$. Hence,
\begin{equation}\label{eq:FUsimplified}
\varphi(U) = \sum_{m = 1}^{2N} |{\mathcal{Z}_U}'(\theta_m)|^{-1} =  2^{1-N}
\sum_{j=1}^N |\sin\theta_j|^{-1} \prod_{\substack{k=1\\ k\neq j}}^{N}|\cos\theta_j-\cos\theta_k|^{-1}.
\end{equation}

To integrate $\varphi(U)$ over $\USp(2N, \mathbb{C})$, we use the Weyl integration formula again. That is, from \eqref{eq:FUsimplified}, \eqref{eq:HaarOnUSp} and \eqref{eq:WeylIntegration}, 
\begin{multline*}
\mathcal{I}(N) = \int_{\USp(2N, \mathbb{C})}\varphi(U)\, \mathrm{d}\mu_{\mathrm{Haar}}(U) \\
= \frac{2^{N^2-N+1}}{N!\pi^N} \int_{[0,\pi]^N} 
\left[
\sum_{j=1}^N |\sin\theta_j|^{-1} \prod_{\substack{k=1\\ k\neq j}}^{N}|\cos\theta_j-\cos\theta_k|^{-1}
\right] \\
\times \prod_{1\le j<k \le N}(\cos\theta_j - \cos\theta_k)^2 \prod_{j=1}^N \sin^2\theta_j \prod_{j=1}^N \mathrm{d}\theta_j.
\end{multline*}
The expression in the square bracket inside of the above integral is symmetric in $\theta_j$'s, therefore, we can replace the summation on $j$ by $N$ times any single summand, say, the $j=N$ term. Then we proceed
\begin{multline}\label{eq:FUintegralFull}
\mathcal{I}(N) 
= \frac{2^{N^2-N+1}}{(N-1)!\pi^N} \int_{[0,\pi]^N} 
\left[
|\sin\theta_N|^{-1} \prod_{k=1}^{N-1}|\cos\theta_N-\cos\theta_k|^{-1}
\right] \\
\times \prod_{1\le j<k \le N}(\cos\theta_j - \cos\theta_k)^2 \prod_{j=1}^N \sin^2\theta_j \prod_{j=1}^N \mathrm{d}\theta_j\\
= \frac{2^{N^2-N+1}}{(N-1)!\pi^N} \int_{[0,\pi]^N} 
\left[
|\sin\theta_N|\prod_{k=1}^{N-1}|\cos\theta_N-\cos\theta_k|
\right] \\
\times \prod_{1\le j<k \le N-1}(\cos\theta_j - \cos\theta_k)^2 \prod_{j=1}^{N-1} \sin^2\theta_j \prod_{j=1}^N \mathrm{d}\theta_j.
\end{multline}
Again, \eqref{eq:TrigIdentity} can be used to rewrite the expression in the square bracket after the last equality in \eqref{eq:FUintegralFull} as
\begin{align*} 
|\sin\theta_N|\prod_{k=1}^{N-1}|\cos\theta_N-\cos\theta_k|
&= |\sin\theta_N| \prod_{k=1}^{N-1}
\frac12 \left|1 - \mathrm{e}^{\mathrm{i}(\theta_k - \theta_N)}\right|\left|1 - \mathrm{e}^{\mathrm{i}(\theta_k + \theta_N)}\right|\\
&= 2^{1-N}
|\sin\theta_N| |\mathcal{Z}_U(\theta_N)|,
\end{align*}
where $U$ is an element in $\USp(2(N-1))$ whose eigenangles are $\pm\theta_1, \dots, \pm\theta_{N-1}$. So, we continue from \eqref{eq:FUintegralFull} to get
\begin{multline}\label{eq:FUintegralFull2}
\mathcal{I}(N)
= \frac{2^{(N-1)^2}}{(N-1)!\pi^N} \int_{[0,\pi]^N} 
2|\sin\theta_N| |\mathcal{Z}_U(\theta_N)|
\\
\times \prod_{1\le j<k \le N-1}(\cos\theta_j - \cos\theta_k)^2 \prod_{j=1}^{N-1} \sin^2\theta_j \prod_{j=1}^N \mathrm{d}\theta_j \\
= 
\frac{2}{\pi}
\int_{[0, \pi]} \sin\theta_N
\left(
\int_{\USp(2(N-1))}
|\mathcal{Z}_U(\theta_N)|\,
\mathrm{d}\mu_{\mathrm{Haar}}(U)
\right)
\mathrm{d}\theta_N,
\end{multline}
where the last equality is again from the Weyl integration formula \eqref{eq:HaarOnUSp} and \eqref{eq:WeylIntegration}, applied to $\USp(2(N-1))$. The integral in the last line of \eqref{eq:FUintegralFull2} within the parenthesis is precisely the integral 
$\langle|\mathcal{Z}_U(\theta_N)|^{2s}\,\rangle_{\USp(2(N-1))} $, with $s = 1/2$, whose asymptotic expression is found in Theorem \ref{thm:MomentThm}.
Therefore, after some simplification, we find that
\begin{equation}\label{eq:Exchange}
\mathcal{I}(N)=
\sqrt2 G({\textstyle\frac12})^2N^{1/4}\int_0^{\pi} (\sin\theta)^{1/4}(1+o(1))\, \mathrm{d}\theta,
\end{equation}
as $N\to\infty$. 

\begin{remark}\label{rmk:Error}
In order to prove the asymptotic expression \eqref{eq:AnalogOfB}, we let $N\to\infty$ in \eqref{eq:Exchange}. 
Then, if we can exchange the limit and the integral, then the integral of $(\sin\theta)^{1/4}$ is expressed in terms of the beta function and then this would finish the proof of \eqref{eq:AnalogOfB}.
The key step here, therefore, is to estimate the size of $o(1)$-term in \eqref{eq:Exchange} with respect to $\theta$. 

This error term comes from the formula of Deift, Its, and Krasovsky, quoted in \eqref{eq:Deift}. See Remark 1.6 of \cite{DIK11} for some general discussion on the size of their error term. 
Let $\epsilon_N(\theta)$ be the $o(1)$-term in \eqref{eq:Exchange}. If one can show that the expression $(\sin\theta)^{1/4}\epsilon_N(\theta)$ is bounded by a function in $L^1([0, \pi])$ independently of $N$, then the dominated convergence theorem can be used to justify the exchange of limit and integral.

In fact, we can show that $\epsilon_N(\theta)$ does \emph{not} tend to zero uniformly in $\theta$ as follows. Define
\[
f_N(\theta):=
\int_{\USp(2N, \mathbb{C})} |\mathcal{Z}_U(\theta)|\, \mathrm{d}\mu_{\mathrm{Haar}}(U).
\]
Then, it is known that
\[ f_N(0) \sim N,\]
from a result in \cite{KS00b}. (Or, alternatively, one can use the same formula of Deift, Its, and Krasovsky and proceed exactly as in the proof of Theorem \ref{thm:MomentThm}.)
If we assume that the error term in Theorem \ref{thm:MomentThm} is bounded uniformly in $\theta$, 
we can pick $N$ large enough, so that
$f_N(\theta)$ is about (a constant times) $N^{1/4}(\sin\theta)^{-3/4}$, for all $\theta$ close to 0. Now, if we now choose $\theta$ in the range $0< \theta < 1/N^2$, then 
this gives a contradiction to the continuity of $f_N(\theta)$ at $\theta=0$. 
So, this implies that $\epsilon_N(\theta)$ does not tend to 0 uniformly in $\theta$ as $N\to\infty$. Therefore, further investigation on $\epsilon_N(\theta)$ is warranted to justify the exchange of limit and integral in \eqref{eq:Exchange}. 
\end{remark}

\end{document}